\documentclass{amsart}

\usepackage{enumerate}
\usepackage{amssymb}
\usepackage{graphicx}
\usepackage{MnSymbol}

\usepackage{amsthm}

\title{Bi-orders do not arise from total orders}
\author{Samuel M. Corson}

\theoremstyle{definition}\newtheorem{theorem}{Theorem}
\theoremstyle{definition}
\theoremstyle{definition}
\theoremstyle{definition}

\theoremstyle{definition}\newtheorem{bigtheorem}{Theorem}

\numberwithin{theorem}{section}
\theoremstyle{definition}
\theoremstyle{definition}\newtheorem{proposition}[theorem]{Proposition}
\theoremstyle{definition}
\theoremstyle{definition}\newtheorem{question}[theorem]{Question}
\theoremstyle{definition}
\theoremstyle{definition}
\theoremstyle{definition}
\theoremstyle{definition}\newtheorem{lemma}[theorem]{Lemma}
\theoremstyle{definition}
\theoremstyle{definition}
\theoremstyle{definition}\newtheorem{definitions}[theorem]{Definitions}
\theoremstyle{definition}
\theoremstyle{definition}
\theoremstyle{definition}\newtheorem{construction}[theorem]{Construction}

\newcommand{\Aut}{\operatorname{Aut}}

\newcommand{\fix}{\operatorname{fix}}
\newcommand{\stab}{\operatorname{stab}}

\newcommand{\At}{\mathcal{A}}
\newcommand{\G}{\mathcal{G}}
\newcommand{\F}{\mathbb{F}}

\newcommand{\AC}{\textbf{AC}}
\newcommand{\LG}{\textbf{LG}}
\newcommand{\ZF}{\textbf{ZF}}
\newcommand{\ZFA}{\textbf{ZFA}}
\newcommand{\Ord}{\operatorname{Ord}}
\newcommand{\Mon}{\operatorname{Mon}}
\newcommand{\Len}{\operatorname{Len}}
\newcommand{\Ho}{\mathcal{H}}

\begin{document}

\address{Instituto de Ciencias Matem\'aticas CSIC-UAM-UC3M-UCM, 28049 Madrid, Spain.}

\email{sammyc973@gmail.com}
\thanks{This work was supported by the Severo Ochoa Programme for Centres of Excellence in R\&D SEV-20150554.}

\begin{abstract}  We present some Zermelo-Fraenkel consistency results regarding bi-orderability of groups, as well as a construction of groups with Conradian orders whose every action on metric spaces has bounded orbits.  A classical consequence of the ultrafilter lemma is that a group is bi-orderable if and only if it is locally bi-orderable.  We show that there exists a model of ZF in which there is a group which is locally free (ergo locally bi-orderable) and not bi-orderable, and the group can be given a total order.  Such a group can also exist in the presence of the principle of dependent choices.  Comparable consistency results are provided for torsion-free abelian groups.
\end{abstract}

\maketitle

\begin{section}{Introduction}

The goals of this paper are to explore the set theoretic strength of bi-orderability in the setting of Zermelo-Fraenkel set theory and also to construct unusual locally indicable groups.  Let $\ZF$ denote Zermelo-Fraenkel set theory minus $\AC$, the axiom of choice.  Recall that a \emph{total order} on a set $X$ is a binary relation $<$ for which exactly one of  $x < y$ or $y < x$ holds for distinct $x, y\in X$, $x < x$ is false for all $x\in X$, and $x < y$ and $y < z$ imply $x < z$.

If $G$ is a group we say that a total order $<$ on $G$ is a \emph{left-order} (respectively \emph{right-order}) provided  for all $g, h, k\in G$ we have that $g < h$ implies $kg < kh$ (resp. $gk < hk$).  We say $G$ is \emph{left-orderable} provided there exists a left order on $G$.  One could similarly define \emph{right-orderable} but since a left-order explicitly defines a right-order and vice-versa, questions of left- or right-orderability of a group are equivalent.  Left-orderable groups are torsion-free.  A group order is a \emph{bi-order} if it is both a left- and right-order and a group is \emph{bi-orderable} provided such an order exists.

The ultrafilter lemma (every filter on a set extends to an ultrafilter) implies the classically known local-to-global bi-orderability result (see \cite[Proposition 1.4]{N}):

\begin{center} \emph{A group $G$ is bi-orderable if and only if every finitely generated subgroup is bi-orderable.}
\end{center}

In a nice setting one can have explicit bi-orders without having recourse to this local-to-global theorem.  Given a total order on a set $X$ one immediately obtains a bi-order on the free abelian group $F_{ab}(X)$ generated by $X$ by considering the lexicographic order, and a bi-order on a free abelian group restricts to a total order on the free set of generators.  Importantly the assertion that every set can be given a total order cannot be proved from $\ZF$, so a total order on an arbitrary set $X$ does not exist a priori.  Thus in $\ZF$ a free abelian group is left orderable if and only if a free set of generators can be given a total order.  By a more elaborate argument, in $\ZF$ a free group is bi-orderable if and only if a free generating set has a total order \cite{M}.  It seems natural to ask whether in $\ZF$ a total order on a locally free, or a torsion-free abelian, group implies bi-orderbility.  It does not, by the following theorem.

\begin{bigtheorem}\label{biorderable}  If $\ZF$ is consistent then there exists a model of $\ZF$ in which the following hold:
\begin{enumerate}

\item  There exists a group $\G$ such that

\begin{enumerate}

\item $\G$ is locally free;

\item there is an increasing sequence $\{\G_n\}_{n\in \omega}$ of retract subgroups which are each bi-orderable with $\bigcup_{n\in \omega} \G_n = \G$;

\item $\G$ can be given a total order;

\item $\G$ is not bi-orderable.

\end{enumerate}
\item  There exists an abelian group $\At$ such that

\begin{enumerate}
\item $\At$ is torsion-free;

\item there is an increasing sequence $\{\At_n\}_{n\in \omega}$ of retract subgroups which are each bi-orderable with $\bigcup_{n\in \omega} \At_n = \At$;

\item $\At$ can be given a total order;

\item $\At$ is not bi-orderable.
\end{enumerate}

\end{enumerate}
\end{bigtheorem}

By a total order on a group we mean, of course, a total order on the group's underlying set.  Also, the notation $\bigcup_{n\in \omega} \G_n = \G$ is the common shorthand expressing that the underlying set for $\G$ is the union of the underlying sets of the subgroups $\G_n$.

Recall that the \emph{principle of dependent choices} is the assertion that if $R$ is a binary relation on a nonempty set $X$ for which $(\forall x\in X)(\exists y \in X)[x R y]$ then there exists a sequence $\{x_n\}_{n\in \omega}$ for which $x_n R x_{n+1}$.  This principle, which is a consequence of the axiom of choice, implies many of the standard results in analysis and also implies the axiom of countable choices.  We have the following:

\begin{bigtheorem}\label{dependentchoices} If $\ZF$ is consistent then there exists a model of $\ZF$ in which the following hold:
\begin{enumerate}

\item  There exists a group $\G$ which is locally free and can be given a total order, but $\G$ is not bi-orderable.

\item  There exists a torsion-free abelian group $\At$ which can be given a total order, but $\At$ is not bi-orderable.

\item  The principle of dependent choices.

\end{enumerate}

\end{bigtheorem}

We note that Theorem \ref{dependentchoices} gives the independence of \cite[Form 227]{HR} and of \cite[Form 228]{HR} from $\ZF$ plus the principle of dependent choices.  The overall strategy in these independence proofs is to work in permutation models of set theory, constructing the claimed groups via presentations, and using the permutations of the model to eliminate any possibility of a bi-order.

We turn now from consistency results and will assume $\ZF + \AC$ (Zermelo-Fraenkel set theory with the axiom of choice) for the remainder of this discussion.  A group $G$ is \emph{strongly bounded} if whenever $G$ acts by isometries on a metric space every orbit is bounded.  Examples of such groups include the full permutation group on a set \cite{B} and products of finite perfect groups \cite{dC}.  Torsion-free examples were provided by Droste and Holland: the automorphism group of a total order which is doubly transitive \cite{DH}.  For example $\Aut(\mathbb{Q}, <)$ is strongly bounded, and this group is also left-orderable.

A group is \emph{locally indicable} if every nontrivial finitely generated subgroup has a nontrivial map to the group $\mathbb{Z}$.  A left-order $<$ on a group $G$ is said to be \emph{Conradian} if for each $g, h\in G$ with $1_G < g, h$ there exists some $n\in \omega$ such that $g < h g^n$.  A group is locally indicable if and only if it has a Conradian left-order \cite[Section 3.3]{N}.  The group $\Aut(\mathbb{Q}, <)$ is not Conradian since every countable left-orderable group embeds into it (see proof of \cite[Proposition 2.1]{N}), and there exist countable left-orderable groups which are not Conradian \cite{B2}.  We construct strongly bounded groups whose local properties are stricter than those of $\Aut(\mathbb{Q}, <)$.

\begin{bigtheorem}\label{locallyindicable}  If $K$ is locally indicable then there exists a simple, locally indicable, strongly bounded group $G \geq K$ with $|G| = |K|^{\aleph_0}$.
\end{bigtheorem}

In Section \ref{prelim} we will give some background on permutation models and transfer theorems.  A reader who is familiar with these techniques and recognizes that the existence of such groups as are claimed in Theorems \ref{biorderable} and \ref{dependentchoices} is boundable may safely skip this material.  We give the proofs of  Theorems \ref{biorderable} and \ref{dependentchoices} in Section \ref{TheoremAB}, and the proof of Theorem \ref{locallyindicable} in Section \ref{TheoremD}.  Concluding remarks are given in Section \ref{concludingremarks}.
\end{section}

\begin{section}{Permutation Model Preliminaries}\label{prelim}

In proving Theorems \ref{biorderable} and \ref{dependentchoices} we will make use of permutation models of Zermelo-Fraenkel set theory with atoms ($\ZFA$), of which we now give a very brief review.   For more details the reader can refer to \cite[Chapter 4]{J}.  The theory $\ZFA$ is a modification of $\ZF$ which allows for a nonempty set $A$ of objects called \emph{atoms} which are themselves not sets.  This requires a tweaking of the axioms of regularity and extensionality.  For a \emph{permutation model} of $\ZFA$ we begin with a model $\mathcal{M}$ of $\ZFA$ with set $A$ of atoms.  Let $\Gamma$ be a group of permutations on the set $A$.  The action of $\Gamma$ on $A$ extends to an action on all of $\mathcal{M}$ in the natural way, by $\epsilon$-induction.  For each $B \subseteq A$ we let $\fix(B) = \{\tau \in \Gamma\mid (\forall a\in B)\tau(a) = a\}$ and for an object $x \in \mathcal{M}$ we let $\stab(x) = \{\tau\in \Gamma \mid \tau(x) = x\}$.

A \emph{normal filter $\mathcal{F}$ of subgroups of $\Gamma$} is a collection of subgroups of $\Gamma$ such that 

\begin{enumerate}

\item $\Gamma \in \mathcal{F}$;

\item if $\tau\in \Gamma$ and $H\in \mathcal{F}$ then $\tau H\tau^{-1}\in \mathcal{F}$;

\item if $H_0\in \mathcal{F}$ and $H_0\leq H_1 \leq \Gamma$ then $H_1 \in \mathcal{F}$;

\item if $H_0, H_1\in \mathcal{F}$ then $H_0 \cap H_1 \in \mathcal{F}$; and

\item $\fix(\{a\}) \in \mathcal{F}$ for each $a\in A$.

\end{enumerate}

\noindent A natural way of producing a normal filter on $\Gamma$ is from a \emph{normal ideal on $A$}: a collection $\mathcal{I}$ of subsets of $A$ such that

\begin{enumerate}

\item $\emptyset \in \mathcal{I}$;

\item if $B\in \mathcal{I}$ and $\tau \in \Gamma$ then $\tau(B) \in \mathcal{I}$;

\item if $B_0\in \mathcal{I}$ and $B_1 \subseteq B_0$ then $B_0\in \mathcal{I}$;

\item if $B_0, B_1 \in \mathcal{I}$ then $B_0 \cup B_1 \in \mathcal{I}$; and

\item $\{a\} \in \mathcal{I}$ for each $a\in A$.

\end{enumerate}

\noindent The normal filter given by $\mathcal{I}$ is $\{K \leq \Gamma \mid \fix(B) \leq K \text{ for some }B\in \mathcal{I}\}$.  We'll say an object $x$ is \emph{$\mathcal{F}$-symmetric} if $\stab(x) \in \mathcal{F}$, and that $B \in \mathcal{I}$ \emph{supports} $x$ provided $\fix(B) \leq \stab(x)$.  The permutation model $\mathcal{N} \subseteq \mathcal{M}$ given by $(\mathcal{M}, \Gamma, \mathcal{F})$ is the collection of hereditarily $\mathcal{F}$-symmetric objects- those objects $x \in \mathcal{M}$ such that each object in the transitive closure of $x$ is $\mathcal{F}$-symmetric.  This $\mathcal{N}$ is also a model of $\ZFA$ (see \cite[Theorem 4.1]{J}).

Not every consistency result in the $\ZFA$ setting can be made to hold in the $\ZF$ setting, but there are metatheorems which allow for $\ZFA$ consistency results of a certain form to transferred to $\ZF$.  We describe sufficient conditions under which a transfer can be achieved, mostly following the exposition in \cite[Note 103]{HR}.  If $\mathcal{M}$ is a model of $\ZFA$ and $x$ is an object in $\mathcal{M}$ we define

\begin{center}  $R_0(x) = x$

$R_{\alpha + 1}(x) = P(R_{\alpha}(x))\cup R_{\alpha}(x)$

$R_{\alpha}(x) = \bigcup_{\beta < \alpha} R_{\alpha}(x)$ for $\alpha$ a non-zero limit ordinal.

\end{center}

\noindent where $P(X)$ denotes the powerset of $X$.

\begin{definitions}  Suppose $V$ is a model of ZF and $\mathcal{M} \subseteq V$ is a substructure which is a model of ZFA such that $\mathcal{M}$ and $V$ have the same class $\Ord$ of ordinals, the same cofinality function, and the same Hartogs and Lindenbaum numbers, as well as the same cardinality function wherever it is defined on sets in $\mathcal{M}$.

\noindent Let $\overline{x} = (x_0, \ldots, x_{n-1})$ be a tuple of variables.  A formula $\Phi(\overline{x})$ is \emph{absolute for $\mathcal{M}, V$} if for all $j_0, \ldots, j_{n-1} \in \mathcal{M}$ we have

\begin{center} $\Phi^{V}(j_0, \ldots, j_{n-1}) \Longleftrightarrow \Phi^{\mathcal{M}}(j_0, \ldots, j_{n-1})$.
\end{center}

\noindent A formula $\Phi$ is \emph{absolute} if $\Phi$ is absolute for $\mathcal{M}, V$ whenever $\mathcal{M}$ and $V$ are as above.  Similarly a term $\sigma(\overline{x})$ is \emph{absolute} provided $\sigma^{V}(j_0, \ldots, j_{n-1}) = \sigma^{\mathcal{M}}(j_0, \ldots, j_{n-1})$ for all $V, \mathcal{M}$ as above.  A term is \emph{ordinal valued} if $(\forall\overline{x})\sigma(\overline{x}) \in \Ord$ is a theorem of $\ZFA$.  A formula $\Phi(\overline{x})$ is \emph{boundable} if there exists an absolute ordinal valued term $\sigma$ for which the biconditional statement

\begin{center} $\Phi(\overline{x}) \Longleftrightarrow \Phi^{R_{\sigma}(x_0\cup \cdots x_{n-1})}(\overline{x})$
\end{center}

\noindent is a theorem of $\ZFA$.  A statement $\Phi$ is \emph{boundable} if it is the existential closure of a boundable formula.
\end{definitions}

We will make use of the following (see \cite[Theorem 4]{P} or \cite[page 286]{HR}).

\begin{theorem}\label{transfer}  Suppose $\Psi$ is a conjunction of any of the following kinds of statements:

\begin{enumerate}

\item boundable statements;

\item the principle of dependent choices.

\end{enumerate}

\noindent If $\Psi$ has a $\ZFA$ model then $\Psi$ has a $\ZF$ model.
\end{theorem}

We'll show that the claim in Theorem \ref{biorderable}(1) is boundable, and the reasoning for the boundability of the other relevant statements is no more complicated than this, and so Theorem \ref{transfer} implies that it is sufficient to prove Theorems \ref{biorderable} and \ref{dependentchoices} in the $\ZFA$ setting.  Let $\Theta_0(G, \circ_{\G}, ^{-1}, 1_{\G})$ express that $(G, \circ_{\G}, ^{-1}, 1_{\G})$ is a group.  Let $\Theta_1(G, \circ_{\G}, ^{-1}, 1_{\G}, \sigma)$ express that $\Theta_0(G, \circ_{\G}, ^{-1}, 1_{\G})$ and that $\sigma: \omega \rightarrow P(G)$ is such that $\sigma(n)\subseteq \sigma(n+1)$ and $(\sigma(n), \circ_{\G}\upharpoonright \sigma(n), ^{-1}\upharpoonright \sigma(n), 1_{\G})$ is a retract subgroup of $\G$.  Let $\Theta_2(H, \circ_{\Ho}, ^{-1}, 1_{\Ho})$ if and only if $\Theta_0(H, \circ_{\Ho}, ^{-1}, 1_{\Ho}))$ and 

\begin{center}
$(\forall m\in \omega)(\forall f: m \rightarrow H)\newline [f(m)\text{ generates a free subgroup of }(H, \circ_{\Ho}, ^{-1}, 1_{\Ho})]$.
\end{center}

\noindent Let $\Theta_3(H, \circ_{\Ho}, ^{-1}, 1_{\Ho})$ if and only if $\Theta_0(H, \circ_{\Ho}, ^{-1}, 1_{\Ho})$ and

\begin{center}  $(\exists < \subseteq H \times H)[< \text{ is a bi-order on }(H, \circ_{\Ho}, ^{-1}, 1_{\Ho})]$.
\end{center}

\noindent Let $\Theta_4(X)$ signify

\begin{center}$(\exists < \subseteq X \times X)[< \text{ is a total order on }X]$.
\end{center}

Finally, let $\Phi(G, \circ_{\G}, ^{-1}, 1_{\G}, \sigma)$ be the conjunction of

\begin{center}
$\Theta_1(G, \circ_{\G}, ^{-1}, 1_{\G}, \sigma)$
\end{center}

\noindent and

\begin{center}  $\Theta_2(G, \circ_{\G}, ^{-1}, 1_{\G}, \sigma)$
\end{center}

\noindent and

\begin{center}  $(\forall n\in \omega)\Theta_3(\sigma(n), \circ_{\G}\upharpoonright \sigma(n), ^{-1}\upharpoonright \sigma(n), 1_{\G})$
\end{center}

\noindent and

\begin{center} $\Theta_4(G)$
\end{center}

\noindent and

\begin{center}  $\neg \Theta_3(G, \circ_{\G}, ^{-1}, 1_{\G})$.
\end{center}

\noindent It is clear that the existential closure of $\Phi$ is the statement in Theorem \ref{biorderable}(1).  Our choices of subscript $\alpha$ in $R_{\alpha}$ in what follows will be generous and are not intended to be sharp.  Any retraction $(G, \circ_{\G}, ^{-1}, 1_{\G}) \rightarrow (\sigma(n), \circ_{\G}\upharpoonright \sigma(n), ^{-1}\upharpoonright \sigma(n), 1_{\G})$ must lie inside of $R_{5}(G\cup \circ_{\G}\cup ^{-1}\cup 1_{\G}\cup \sigma)$.  We have $\omega \in R_{\omega +1}(\emptyset) \subseteq R_{\omega+1}(G\cup \circ_{\G}\cup ^{-1}\cup 1_{\G}\cup \sigma)$, and every function $f:n \rightarrow G$, with $n\in\omega$, must be in $R_{\omega + 6}(G\cup \circ_{\G}\cup ^{-1}\cup 1_{\G}\cup \sigma)$.  As well any function $f: \omega \rightarrow P(G)$ must be in $R_{\omega + 6}(G\cup \circ_{\G}\cup ^{-1}\cup 1_{\G}\cup \sigma)$.  Also, the abstract free group $\F(X)$ on a set $X$ will lie inside $R_{\omega + 1}(X)$ (see \cite[Example 2.15]{Kl}), and so an isomorphism between a generated subgroup of the image of an arbitrary function $f: m \rightarrow G$ and the abstract free group on a subset thereof will be in, say, $R_{\omega + \omega + 5}(G\cup \circ_{\G}\cup ^{-1}\cup 1_{\G}\cup \sigma)$.  A total order on a set $X$ will be in $R_5(X)$.  Thus all total orderings on $G$ and on each $\sigma(m)$ will certainly be in $R_{\omega + \omega + 10}(G\cup \circ_{\G}\cup ^{-1}\cup 1_{\G}\cup \sigma)$.  So, from $\ZFA$ we have

\begin{center} $\Phi(G, \circ_{\G}, ^{-1}, 1_{\G}, \sigma) \Longleftrightarrow \Phi^{R_{\omega + \omega + 10}(G \cup \circ_{\G}\cup ^{-1} \cup 1_G \cup \sigma)}(G, \circ_{\G}, ^{-1}, 1_{\G}, \sigma)$
\end{center}

\noindent and boundability is proved.

\end{section}

\begin{section}{Theorems \ref{biorderable} and \ref{dependentchoices}}\label{TheoremAB}

For Theorem \ref{biorderable} we will work in a model of van Douwen (see \cite{v} or \cite[Model $\mathcal{N}2(LO)$]{HR}).  We let $\mathcal{M}$ be a model of $\ZFA + \AC$ with a countable set $A$ of atoms.  Write $A$ as a disjoint union $A = \bigcup_{n\in \omega} A_n$ with each $A_n$ being countably infinite.  Endow each $A_n$ with an order $<_n$ which is isomorphic to that of the set $\mathbb{Z}$ of integers.  For $a\in A_n$ let $s(a)$ denote the next largest element in $A_n$ under the order $<_n$.  We let $\Gamma$ be the group of all permutations $\tau$ of $A$ such that $\tau \upharpoonright A_n\in \Aut(A_n, <_n)$ for each $n\in \omega$.  Let $\mathcal{F}$ be the normal filter of subgroups of $\Gamma$ generated by the ideal of finite subsets of $A$.  Let $\mathcal{N}$ be the permutation model given by $(\mathcal{M}, \Gamma, \mathcal{F})$.

\begin{subsection}{Theorem \ref{biorderable} (1).}  Let $J = A\times\{0, 1\}$ and $\F(J) = (W_J, \circ_J, ^{-1}, 1_{W_J})$ denote the free group on the set $J$, with $W_J$ denoting the set of reduced words over the alphabet $J^{\pm1}$,$\circ_J$ and $^{-1}$ denoting the group multiplication and group inversion operations, and $1_{W_J}$ denoting the trivial element.  This group, which we have defined in $\mathcal{M}$, is clearly in $\mathcal{N}$ as well; moreover, $\stab(W_J) = \stab(\circ_J) = \stab(^{-1}) = \Gamma$.  Notice that the subset $X_J = \{(s(a), 0)(a, 1)(s(a), 0)^{-1}(s(a), 1)\}_{a\in A} \subseteq W_J$ is also in $\mathcal{N}$ and also supported by $\emptyset \subseteq A$.  Therefore the normal subgroup $N_J = \langle\langle X_J \rangle\rangle \unlhd \F(J)$ is in $\mathcal{N}$ and supported by $\emptyset$, and the similar claims hold for the quotient $\G =  \F(J)/N_J$.  We emphasize that the identity element $N_J$ of $\G$, which we'll denote $1_{\G}$, is supported by $\emptyset$.

\noindent \textbf{(1) (a) \& (b).}  For each $n\in \omega$ let $J_n = (\bigcup_{0 \leq i \leq n} A_i) \times \{0, 1\}$.  Similarly define the free group $\F(J_n)$ and notice that $\F(J_n)$ is in $\mathcal{N}$ and the set of reduced words in $J_n$, the group multiplication operation and the inverse operation are supported by $\emptyset$.  Let $r_n: \F(J) \rightarrow \F(J_n)$ denote the retraction map given by deleting all letters in $J \setminus J_n$ and freely reducing, and notice that $r_n$ is in $\mathcal{N}$ and supported by $\emptyset$.  Letting $\G_n = F(J_n)N_J$ we see that $\G_n$ is also in $\mathcal{N}$ and supported by $\emptyset$.  Also, $r_n(X_J) \subseteq X_J \cup\{1_{W_J}\}$ and so $r_n(N_J) \subseteq N_J$.  Then the retraction homomorphism $\G \rightarrow \G_n$ given by taking a coset $K$ of $N_J$ to $r_n(K)N_J$ is in $\mathcal{N}$ and is similarly invariant under $\Gamma$.  Notice as well that the function $\{(n, \G_n)\}_{n\in \omega}$ is in $\mathcal{N}$ and supported by $\emptyset$.  We have thus far established the existence of the sequence of retract subgroups $\{\G_n\}_{n\in \omega}$ of $\G$ with $\G = \bigcup_{n\in \omega} \G_n$.

We will show that each $\G_n$ is locally free and bi-orderable, and this is sufficient for (a) and (b) since any finitely generated subgroup of $\G$ includes into some $\G_n$.  Fix $n\in \omega$.  Let $T_n$ denote the group $$\F(J_n)/\langle\langle \{(s(a), 0)(a, 1)(s(a), 0)^{-1}(s(a), 1)\}_{a\in \bigcup_{0 \leq i \leq n}A_i}\rangle\rangle.$$  It is easy to see that $T_n$ is in $\mathcal{N}$ and that the identity map on the generators induces an isomorphism with $\G_n$ (and this isomorphism is also in $\mathcal{N}$).  We establish that $T_n$ is locally free and bi-orderable.

By selecting $a_i \in A_i$ for each $0 \leq i \leq n$ we have $\fix(\{a_0, \ldots, a_n\}) = \fix(\bigcup_{0 \leq i \leq n}A_i)$.  Since the object $T_n$ is hereditarily supported by $\fix(\{a_0, \ldots, a_n\})$ and $\mathcal{M}$ is a model of $\ZFA + \AC$, we may use $\AC$ in arguing that $T_n$ is locally free and bi-orderable.  Since $\AC$ implies that locally free groups are bi-orderable, we therefore need only show that $T_n$ is locally free.  It is clear that $T_n$ is the free product of $n + 1$ copies of the group $H$ given by presentation.

$$
\begin{array}{ll}  \langle \{x_m\}_{m\in \mathbb{Z}} \cup \{y_n\}_{n\in \mathbb{Z}}\mid \{y_n = x_{n+1}^{-1}y_{n+1}^{-1}x_{n+1}\}_{n\in \mathbb{Z}}\rangle

\end{array}\eqno{(1)}
$$

Since the class of locally free groups is closed under taking free products, we now need to show that $H$ is locally free.

\begin{lemma}\label{localfree}  The group $H$ is locally free and all generators $\{x_m\}_{m\in \mathbb{Z}} \cup \{y_n\}_{n\in \mathbb{Z}}$ are nontrivial elements in $H$.
\end{lemma}

\begin{proof}  Notice that for a fixed $N \in \mathbb{Z}$ the presentation defining $H$ does not require the relators $\{y_n = x_{n+1}^{-1}y_{n+1}^{-1}x_{n+1}\}_{n<N}$ and the generators $\{y_n\}_{n<N}$ since the relators $\{y_n = x_{n+1}^{-1}y_{n+1}^{-1}x_{n+1}\}_{n<N}$ are only used in giving names to the elements $\{y_n\}_{n<N}$.  This is because for any positive $k\in \omega$ we can write

\begin{center}
$y_{N-k} = x_{N - k + 1}^{-1}x_{N - k + 2}^{-1}\cdots x_N^{-1} y_N^{(-1)^k}  x_N \cdots x_{N - k + 2} x_{N - k + 1}$
\end{center}

\noindent In particular, for any fixed $N\in \mathbb{Z}$ we know that $H$ is isomorphic to the group $H_N$ with presentation

$$
\begin{array}{ll}  \langle \{x_m\}_{m\in \mathbb{Z}} \cup \{y_n\}_{n\geq N}\mid \{y_n = x_{n+1}^{-1}y_{n+1}^{-1}x_{n+1}\}_{n\geq N}\rangle
\end{array}\eqno{(2)}
$$

\noindent via the map $\rho_N$ determined by

\begin{center}$x_m \mapsto x_m$ for all $n\in\mathbb{Z}$

$y_n \mapsto y_n$ for $n\geq N$

$y_{N - k} \mapsto  x_{N - k + 1}^{-1}x_{N - k + 2}^{-1}\cdots x_N^{-1} y_N^{(-1)^k}  x_N \cdots x_{N - k + 2} x_{N - k + 1}^{-1}$ for $k \geq 1$

\end{center}

\noindent Consider the normal subgroup $K = \langle\langle \{x_n\}_{n>N}\rangle\rangle \unlhd H_N$.  The quotient $H_N/K$ has presentation

$$
\begin{array}{ll}  \langle \{x_m\}_{m\leq N} \cup \{y_n\}_{n\geq N}\mid \{y_n = y_{n+1}^{-1}\}_{n\geq N}\rangle
\end{array}
$$

\noindent and this group is simply the free group in the generators $\{x_m\}_{m\leq N} \cup \{y_N\}$.  This implies that for each $N\in \mathbb{Z}$ the set $\{x_m\}_{m \leq N} \cup \{y_N\}$ freely generates a subgroup of $H_N$.

For any finite set of words $\{w_0, \ldots, w_r\}$ in the letters $\{x_m\}_{m\in \mathbb{Z}}^{\pm 1} \cup \{y_n\}_{n\in \mathbb{Z}}^{\pm 1}$ there exists some $N$ for which each of the words $w_0, \ldots, w_r$ is written in the letters $\{x_m\}_{m\leq N}^{\pm 1} \cup \{y_n\}_{n\leq N}^{\pm 1}$.  Then applying $\rho_N$ to the group elements represented by $w_0, \ldots, w_r$ places this set within the subgroup $\langle \{x_m\}_{m\leq N} \cup \{y_N\}\rangle \leq H_N$, and since this subgroup is free, we have that $H$ is locally free.  The second claim follows immediately from our proof since we showed that for each $N\in \mathbb{Z}$ the set $\{x_m\}_{m\leq N}\cup\{y_N\}$ freely generates a subgroup of $H$.
\end{proof}

\noindent \textbf{(1) (c).} Towards producing a total order on $\G$ we produce, in $\mathcal{M}$, a normal form for $\G$.  Since $\AC$ holds in $\mathcal{M}$ we shall freely use choices in this construction, and the fact that the normal form is also in $\mathcal{N}$ will become apparent.  Recall that a word rewriting system on a free monoid $\Mon(X)$ on set $X$ is a set of rules $\mathcal{R}$ whose inputs and outputs are words in the monoid (see \cite[Section 1.7]{Sap}).  We define binary relation $\rightarrow_{\mathcal{R}}$ on $\Mon(X)$ by letting $w_0 \rightarrow_{\mathcal{R}} w_1$ if there exist $v_0, v_1, v_1', v_2 \in \Mon(X)$ with $w_0 \equiv v_0v_1v_2$ and $w_1 \equiv v_0v_1'v_2$ and $(v_1, v_1')\in \mathcal{R}$.  Let $\rightarrow_{\mathcal{R}}^*$ be the smallest transitive binary relation including $\rightarrow_{\mathcal{R}}$ and let $\leftrightarrow_{\mathcal{R}}^*$ denote the smallest equivalence class including $\rightarrow_{\mathcal{R}}^*$.  Rewriting system $\mathcal{R}$ is \emph{confluent} if whenever $w_0 \rightarrow_{\mathcal{R}}^*w_1$ and $w_0 \rightarrow_{\mathcal{R}}^* w_2$ there exists $w_3$ for which $w_1 \rightarrow_{\mathcal{R}}^* w_3$ and $w_2 \rightarrow_{\mathcal{R}}^* w_3$.  It is \emph{locally confluent} if whenever $w_0 \rightarrow_{\mathcal{R}} w_1$ and $w_0 \rightarrow_{\mathcal{R}} w_2$ there exists $w_3$ for which $w_1 \rightarrow_{\mathcal{R}}^* w_3$ and $w_2 \rightarrow_{\mathcal{R}}^* w_3$.

Rewriting system $\mathcal{R}$ is \emph{terminating} if each sequence $w_0 \rightarrow_{\mathcal{R}} w_1 \rightarrow_{\mathcal{R}} w_2 \cdots \rightarrow_{\mathcal{R}} w_n$ must eventually stabilize.  A word $w$ is a \emph{terminus} of $\mathcal{R}$ if $w \rightarrow_{\mathcal{R}} v$ implies $w \equiv v$.  If $\mathcal{R}$ is terminating and locally confluent then it is confluent, and if $\mathcal{R}$ is terminating and confluent then each equivalence class in $\leftrightarrow_{\mathcal{R}}^*$ contains a unique terminus (see \cite[Section 7.1]{Sap}).

We let $\Mon(J^{\pm 1})$ denote the free monoid on the set $\{(a, 0)\}_{a\in A} \cup \{(a, 1)\}_{a\in A} \cup\{(a, 0)^{-1}\}_{a\in A}\cup \{(a,1)^{-1}\}_{a\in A}$, and let $e$ denote the empty word.  Consider the rewriting system $\mathcal{R}$ under which for all $a\in A$ we have rules

\begin{enumerate}[(i)]

\item $(a, 0)(a, 0)^{-1}\mapsto e$ 

\item $(a, 0)^{-1}(a, 0) \mapsto e$

\item $(a, 1)(a, 1)^{-1}\mapsto e$

\item $(a, 1)^{-1}(a, 1)\mapsto e$

\item $(s(a), 0)(a, 1) \mapsto (s(a), 1)^{-1}(s(a), 0)$

\item $(s(a), 0)(a, 1)^{-1} \mapsto (s(a), 1)(s(a), 0)$

\item $(s(a), 0)^{-1}(s(a), 1) \mapsto (a, 1)^{-1}(s(a), 0)^{-1}$

\item $(s(a), 0)^{-1}(s(a), 1)^{-1} \mapsto (a, 1)(s(a), 0)^{-1}$
\end{enumerate}

\noindent The idea of this system is to both freely reduce and to move the $(a, 0)^{\pm 1}$ letters to the right.  We claim that this rewriting system is locally confluent.  We'll sketch the argument in the slightly less obvious cases.

Suppose that we have a word $w \equiv v_0(s(a), 0)^{-1}(s(a), 0)(a, 1)v_1$.  If one first applies (ii) then one obtains the word $v_0(a, 1)v_1$.  If we first use (v) then we get the word $v_0(s(a), 0)^{-1}(s(a), 1)^{-1}(s(a), 0)v_1$, and applying rule (viii) we get word $v_0(a, 1)(s(a), 0)^{-1}(s(a), 0) v_1$, and applying rule (ii) we obtain $v_0(a, 1)v_1$.

Suppose that we have word $w \equiv v_0(s(a), 0)(a, 1)^{-1}(a, 1)v_1$.  If we immediately apply rule (iv) then we obtain $v_0(s(a), 0)v_1$.  If, instead, we first apply (vi) to obtain $v_0(s(a), 1)(s(a), 0)(a, 1)v_1$, then applying rule (v) gives us the word $$v_0(s(a), 1)(s(a), 1)^{-1}(s(a), 0)v_1$$ and rule (iii) gives $v_0(s(a), 0)v_1$.

All other checks for local confluence are either quite trivial or are argued in like manner.  We note also that the rewriting system is terminating.  To see this, given a word $w$ we consider the function 

\begin{center}
$j(w) = \sum_{0 \leq i <\Len(w), w(i) \in \{(a, 0)^{\pm 1}\}_{a\in A} } |\{i < k < \Len(w)\mid w(k) \in  \{(a', 1)^{\pm 1}\}_{a'\in A}\}|$

\end{center}

\noindent which counts the total number of times that a letter of form $(a', 1)^{\pm 1}$ appears in the word somewhere to the right of a letter of form $(a, 0)^{\pm 1}$.  Each application of a rule will lower the value of the function $\Len(w) + j(w)$, and so the fact that the system is terminating follows.  Thus each equivalence class under $\leftrightarrow_{\mathcal{R}}^*$ contains a unique terminus.

All elements of the set $R$ of words which are the terminus of a word in $\Mon(J^{\pm 1})$ under $\mathcal{R}$ are freely reduced. The set $R$ is also obviously in $\mathcal{N}$ (notice that the rules are themselves invariant under the action of $\Gamma$) and supported by $\emptyset$.  Furthermore it is straightforward to see that each element in $R$ is a unique representative of an element in $\G$.  We give an order $<^l$ to the letters in $J^{\pm 1}$ as follows:

\begin{center}

$(a, 0)^{-1} <^l (a, 0) <^l (a, 1)^{-1} <^l (a, 1) <^l (a', 0)^{-1} <^l (a', 0) <^l (a', 1)^{-1} <^l (a', 1)$

\end{center}

\noindent where either $a, a'\in A_n$ with $a <_n a'$ or $a\in A_n$ and $a'\in A_{n'}$ with $n < n'$.   Endow the elements of $R$ with the shortlex order $<^o$: $w_0 <^o w_1$ if either $\Len(w_0) < \Len(w_1)$, or $\Len(w_0) = \Len(w_1)$ and for the least $0 \leq i <\Len(w_0)$ at which $w_0(i) \neq w_1(i)$ we have $w_0(i) <^l w_1(i)$.  It is clear that both $<^l$ and $<^o$ are in $\mathcal{N}$, and more particularly they are supported by $\emptyset$.

\noindent \textbf{(1) (d).}  To see that $\G$ is not bi-orderable we suppose for contradiction that $<_{\G}$ is a bi-order on $\G$ in $\mathcal{N}$.  Select finite $B \subseteq A$ for which $\fix(B) \leq \stab(<_{\G})$.  Select $n\in \omega$ large enough that $A_n \cap B = \emptyset$.  Let $\tau \in \Gamma$ be given by

\begin{center} $\tau(a) = \begin{cases}a$ if $a\notin A_n\\ s(a) $ if $a\in A_n   \end{cases}$
\end{center}

\noindent  Let $a\in A_n$ be given.  By Lemma \ref{localfree} we know that $(a, 1)N_J$ is nontrivial.  If $1_{\G} <_{\G} (a, 1)N_J$ then $1_{\G} <_{\G} (s(a), 0)(a, 1)(s(a), 0)^{-1}N_J = (s(a), 1)^{-1}N_J$, from which we see that $(s(a), 1)N_J <_{\G} 1_{\G}$, but on the other hand

\begin{center}  $1_{\G} = \tau(1_{\G}) <_{\G} \tau((a, 1))N_J = (s(a), 1)N_J$
\end{center} 

\noindent which is a contradiction.  The proof in case $(a, 1)N_J <_{\G} 1_{\G}$ is symmetric.
\end{subsection}

\begin{subsection}{Theorem \ref{biorderable} (2).}  We sketch over the aspects of the proof which are nearly identical to those in  (1).  We take $\F(A)$ to be the free group on the set $A$ of atoms.  Consider the subset $X_A = \{[a, a']\}_{a, a'\in A} \cup \{a(s(a))^2\}_{a\in A}$, where $[a, a']$ denotes the commutator $aa'a^{-1}(a')^{-1}$.  This set is in $\mathcal{N}$ and supported by $\emptyset$, and similarly for the relevant group operations and underlying set of $\At = \F(A)/\langle\langle X_A\rangle\rangle$.  Letting $0_{\At}$ denote the identity element, we emphasize that $0_{\At}$ is supported by $\emptyset$.  Let $B_n = \bigcup_{0 \leq i \leq n}A_n$ and $r_n: \F(A) \rightarrow \F(B_n)$ be the retraction.  Let $\At_n = \F(B_n)N_A$.  Since $r_n(Y) \subseteq Y \cup \{1\}$ we have $r_n(N_A) \subseteq N_A$.  Thus we have a retraction map $\At \rightarrow \At_n$ given by $K \mapsto r_n(K) N_A$ which is in $\mathcal{N}$ and supported by $\emptyset$.

Define $L_n$ by $\F(B_n)/\langle \langle \{[a, a']\}_{a, a'\in B_n} \cup \{a(s(a))^2\}_{a\in B_n} \rangle\rangle$.  Notice that $L_n \simeq \At_n$ via the identity map on the generators (and this isomorphism is in $\mathcal{N}$).  Taking $a_i \in A_i$ for each $0 \leq i \leq n$ we have again that $\fix(\{a_0, \ldots, a_n\}) = \fix(B_n)$.  Thus we may utilize $\AC$, which holds in $\mathcal{M}$, in analyzing $L_n$.  It is easy to see that $L_n$ is isomorphic to a direct sum of $n+1$ copies of the additive group $\mathbb{Z}[\frac{1}{2}]$.  Thus $\At_n$, and therefore all of $\At$, is torsion-free and for each $a\in A$ we have $aN_A$ nontrivial in $\At$.

A normal form on $\At$ is given by words of the form

\begin{center} $a_0^{z_0}a_1^{z_1}\cdots a_m^{z_m}$
\end{center}

\noindent where for each $ 0\leq i \leq m$ we have $z_i \in \mathbb{Z} \setminus 2\mathbb{Z}$ and $a_i\in A_{j_i}$ with $j_0 < j_1 < \cdots < j_m$.  The set of all such words is in $\mathcal{N}$ and supported by $\emptyset$.  Order the letters $A^{\pm 1}$ by order $<^l$ given by

\begin{center}  $a^{-1} <^l a <^l (a')^{-1} <^l a'$
\end{center}

\noindent where $a, a'\in A_n$ for some $n\in \omega$ and $a <_n a'$ or $a\in A_n$ and $a'\in A_{n'}$ with $n < n'$.  This order $<^l$ is invariant under $\Gamma$.  Order $\At$ using shortlex on the normal form.

Now suppose that $<_{\At}$ is a bi-order on $\At$.  Let $B \subseteq A$ be finite with $\fix(B) \leq \stab(<_{\At})$.  Select $n \in \omega$ such that $A_n \cap B = \emptyset$.  Let $\tau \in \Gamma$ be given by

\begin{center}  $\tau(a) = \begin{cases} a$ if $a\in A \setminus A_n\\ s(a)$ if $a\in A_n  \end{cases}$.
\end{center}

\noindent Let $a\in A_n$.  Suppose that $0_{\At} <_{\At} aN_A$.  On one hand we have that $0_{\At} = \tau(0_{\At}) <_{\At} \tau(aN_A) = s(a)N_A$, but on the other hand we have $s(a)N_A = a^{-2}N_A <_{\At} 0_{\At}$, a contradiction.  The proof in case $aN_A <_{\At} 0_{\At}$ is symmetric.
\end{subsection}

\begin{subsection}{Theorem \ref{dependentchoices}}  The proof of this theorem follows analogous lines and we give the sketch.  Let $\mathcal{M'}$ be a model of $\ZFA + \AC$ with a set $A'$ of atoms which is of cardinality $\aleph_1$.  Express $A'$ as a disjoint union $A' = \bigcup_{\alpha < \aleph_1} A_{\alpha}'$ with each $A_{\alpha}'$ being countably infinite and endowed with a total order $<_{\alpha}$ which makes $A_{\alpha}'$ order isomorphic to $\mathbb{Z}$.  Let $\Gamma'$ be the set of bijections $\tau$ on $A$ for which $\tau \upharpoonright A_{\alpha}' \in \Aut(A_{\alpha}', <_{\alpha})$ for all $\alpha< \aleph_1$.  Let $\mathcal{F}'$ be the normal filter on $\Gamma'$ given by countable subsets of $A'$.  Define $\G$ in the analogous way substituting $A'$ for $A$ and check local freeness as before but using retractions which erase letters in $((A' \setminus \bigcup_{\alpha <\gamma} A_{\alpha}') \times \{0, 1\})^{\pm 1}$ with $\gamma <\aleph_1$.  Formulate the comparable normal form and order $\G$, use the comparable argument in showing that $\G$ is not bi-orderable.  Define and treat the abelian group $\At$ under the comparable analogy.  That the model $\mathcal{N'}$ satisfies the principle of dependent choices follows from the fact that the ideal defining the filter $\mathcal{F}'$ is closed under taking countable unions (see \cite[Note 144]{HR}).
\end{subsection}

\end{section}

\begin{section}{Theorem \ref{locallyindicable}}\label{TheoremD}

We remind the reader of some basic facts about locally indicable groups, utilizing $\AC$ freely for the entirety of this section.  The class of locally indicable groups includes all free groups, since nontrivial finitely generated subgroups of free groups are free of rank at least $1$ and therefore indicable.  Subgroups of locally indicable groups are obviously locally indicable.  Moreover the class is closed under extensions: if $1 \rightarrow N \rightarrow G \rightarrow Q \rightarrow 1$ is a short exact sequence with $N$ and $Q$ locally indicable then so is $G$, since a finitely generated nontrivial subgroup of $G$ will either lie inside of $N$ or will map nontrivially to $Q$.

The class is also closed under taking free products.  If $A$ and $B$ are locally indicable then so is $A \times B$ and the standard short exact sequence $$1 \rightarrow F \rightarrow A * B \rightarrow A \times B \rightarrow 1$$ with $F$ being a free subgroup of $A * B$, demonstrates that $A * B$ is locally indicable.  By induction this class is closed under free products $*_{i\in I}G_i$ with $I$ finite, and when $I$ is infinite a finitely generated subgroup will lie inside of some $*_{i \in I'}G_i$ with $I' \subseteq I$ finite, so we indeed have closure under arbitrary free products.

We will make use of two results of Karrass and Solitar (see \cite[Theorem 2]{KS} and \cite[Theorem 6]{KS}, respectively).   The setup of these results is the following:  Let $J$ be a group and $\phi_i: A_i \rightarrow B_i$ be a collection of isomorphisms between subgroups $A_i, B_i \subseteq J$.  Let $L$ be the HNN extension $J *_{t_iA_it_i^{-1} = \phi_i(B_i)}$.

\begin{proposition}\label{HNNlocallyindicable}  If $J$ is locally indicable and each of the $A_i$ is cyclic then $L$ is locally indicable.
\end{proposition}

\begin{proposition}\label{HNNsubgroups}  If $H \leq L$ is a subgroup which has trivial intersection with all conjugates of $A_i$ and $B_i$ in $L$ then $H$ is the free product of a free group and the intersections of $H$ with certain conjugates of $J$ in $L$.
\end{proposition}

\begin{construction}\label{beatdown}  Suppose that $M \leq J$ are nontrivial torsion-free groups and that $\sigma: \mathbb{Z} \rightarrow M \setminus \{1\}$ is a function.  Take $L_0$ to be the HNN extension of $J$ given by $L_0 = J*_{{t_z\langle \sigma(z)\rangle t_z^{-1} = \langle \sigma(z+1) \rangle}_{z\in \mathbb{Z}}}$.  Now the free group $F(\{t_z\}_{z\in \mathbb{Z}})$ is a retract subgroup of $L_0$ and we let $\phi$ be the automorphism on $F(\{t_z\}_{z\in \mathbb{Z}})$ for which $\phi(t_z) = t_{z+1}$.  Let $E(M, J, \sigma)$ denote the HNN extension $L_0 *_{t\langle t_z\rangle t^{-1} = \langle t_{z+1}\rangle}$.  This group $E(M, J)$ will also be torsion-free by the standard theorems regarding HNN extensions, and also $J$ naturally embeds as a subgroup of $E(M, J)$.

Given a torsion-free group $J$ we let $\{\sigma_{\alpha}\}_{\alpha < |J|^{\aleph_0}}$ be a well ordering of the functions $\sigma:\mathbb{Z} \rightarrow J \setminus \{1\}$.  We define an increasing sequence $\{J_\alpha\}_{\alpha < |J|^{\aleph_0}}$ of torsion-free nesting groups.  Let $J_0 = J$.  If $J_{\alpha}$ has been defined for all $\alpha<\beta < |J|^{\aleph_0}$ and $\beta = \alpha + 1$ then let $J_{\beta} = E(J, J_{\alpha}, \sigma_{\alpha})$.  If $\beta$ is a limit ordinal then let $J_{\beta} = \bigcup_{\alpha < \beta} J_{\alpha}$.  Let $E(J)$ denote the union $\bigcup_{\alpha < |J|^{\aleph_0}} J_{\alpha}$.  The construction of the group $E(J)\geq J$ formally depended, of course, on the well ordering of the $\sigma$.  The order in which we took these HNN extensions actually does not make any difference up to the isomorphism class of $E(J)$, so the well ordering does not appear in the notation.

\end{construction}

\begin{lemma}\label{obvious}  The group $E(J)$ is locally indicable provided $J$ is.  If $|J| = |J|^{\aleph_0}$ then $|E(J)| = |J|$.
\end{lemma}

\begin{proof}  Since local indicability is preserved under infinite increasing unions, it suffices to show that if $J$ is a locally indicable group, $M \leq J$, and $\sigma: \mathbb{Z} \rightarrow M \setminus \{1\}$ then $E(M, J, \sigma)$ is also locally indicable (by induction).  To see that $E(M, J, \sigma)$ is locally indicable, we first notice that the extension $L_0$ defined in Construction \ref{beatdown} is locally indicable by Proposition \ref{HNNlocallyindicable}.  Next, we let $r: L_0 \rightarrow F(\{t_z\}_{z\in \mathbb{Z}})$ be the natural retraction.  This extends to a retraction $r' : E(M, J, \sigma) \rightarrow F(\{t_z\}_{z\in \mathbb{Z}}) *_{t\langle t_z\rangle t^{-1} = \langle t_{z + 1}\rangle}$ by mapping $t \mapsto t$ and $g \mapsto r(g)$ for $g\in L_0$.  The group $F(\{t_z\}_{z\in \mathbb{Z}}) *_{t\langle t_z\rangle t^{-1} = \langle t_{z + 1}\rangle}$ is a split extension $$1 \rightarrow F(\{t_z\}_{z\in \mathbb{Z}}) \rightarrow F(\{t_z\}_{z\in \mathbb{Z}}) *_{t\langle t_z \rangle t^{-1} = \langle t_{z + 1}\rangle} \rightarrow \langle t\rangle \rightarrow 1$$ and so $F(\{t_z\}_{z\in \mathbb{Z}}) *_{t\langle t_z\rangle t^{-1} = \langle t_{z + 1}\rangle}$ is locally indicable as an extension of two locally indicable groups.

For the kernel $\ker(r') \leq E(M, J, \sigma)$ it is clear that $\ker(r')\cap F(\{t_z\}_{z\in \mathbb{Z}}) *_{t\langle t_z\rangle t^{-1} = \langle t_{z + 1}\rangle}$ is trivial (since $r'$ is a retraction).  Thus more particularly $\ker(r') \cap  F(\{t_z\}_{z\in \mathbb{Z}})$ is trivial.  Since $\ker(r')$ is normal in $ E(M, J, \sigma)$ we know that $\ker(r')$ has trivial intersection with all conjugates in $E(M, J, \sigma)$ of the subgroup $F(\{t_z\}_{z\in \mathbb{Z}})$.  By Proposition \ref{HNNsubgroups} we have that $\ker(r')$ is a free product of a free group and groups which are isomorphic to subgroups of $L_0$.  Thus $\ker(r')$ is locally indicable as a free product of locally indicable groups.  Now $E(M, J, \sigma)$ is locally indicable as an extension $$1 \rightarrow \ker(r') \rightarrow E(M, J, \sigma) \rightarrow F(\{t_z\}_{z\in \mathbb{Z}}) *_{t\langle t_z\rangle t^{-1} = \langle t_{z + 1}\rangle} \rightarrow 1$$ of locally indicable groups.

Suppose that $|J| = |J|^{\aleph_0}$.  We have been assuming that $J$ is nontrivial and so $|J|$ is uncountable.  We see by induction that $|J| \leq|J_{\alpha}| \leq |J|^{\aleph_0}|J|^{\aleph_0} = |J|^{\aleph_0}$ for all $\alpha < |J|^{\aleph_0}$, and so $|E(J)| = |J|^{\aleph_0}$.
\end{proof}

We use a necessary and sufficient criterion for strong boundedness given by de Cornulier (see \cite[Proposition 2.7]{dC}):

\begin{proposition}\label{criterionforstrongboundedness}  A group $G$ is strongly bounded if and only for every function $\Lambda: G \rightarrow \omega$ such that for all $g, h\in G$ we have

\begin{itemize}

\item $\Lambda(1) \leq 1$;

\item $\Lambda(g) \leq \Lambda(g^{-1}) + 1$; and

\item $\Lambda(gh) \leq \max(\Lambda(g), \Lambda(h)) + 1$

\end{itemize}

\noindent there exists some bound $P \in \mathbb{N}$ for which $\Lambda(g)\leq P$ for all $g\in G$.

\end{proposition}

\begin{proof}[Proof of Theorem \ref{locallyindicable}]  Let $K$ be a locally indicable group.  If $K$ is the trivial group then we let $G = K$ and we are done.  Suppose $K$ is nontrivial.  We can assume without loss of generality that $|K| = |K|^{\aleph_0}$ by replacing $K$ with the free product of $K$ with the free group of rank $|K|^{\aleph_0}$.  We define $G$ by an increasing nesting sequence $\{K_{\alpha}\}_{\alpha <\aleph_1}$ of supergroups of $K$.  Let $K_0 = K$.  If $K_{\alpha}$ has been defined for all $\alpha < \beta < \aleph_1$ and $\beta$ is a limit ordinal then let $K_{\beta} = \bigcup_{\alpha <\beta} K_{\alpha}$.  If $\beta = \alpha + 1$ then let $K_{\beta} = E(K_{\alpha})$.  Let $G = \bigcup_{\alpha <\aleph_1} K_{\alpha}$.  Notice that each $K_{\alpha}$ has cardinality $|K|^{\aleph_0}$.

We have $|K| = |K|^{\aleph_0} \leq |G| \leq \aleph_1 \cdot |K|^{\aleph_0}$, and so $G$ has the correct cardinality.  The group $G$ is also locally indicable as an increasing union of locally indicable groups $K_{\alpha}$ by Lemma \ref{obvious}.  That $G$ is simple follows from the fact that any two nontrivial elements are conjugate.  More particularly, given $g, h\in G\setminus \{1\}$ we select $\alpha < \aleph_1$ for which $g, h \in K_{\alpha}$, let $\sigma: \mathbb{Z} \rightarrow K_{\alpha}\setminus \{1\}$ be given by 
\begin{center}
$\sigma(n) = \begin{cases} g$ if $n\geq 0\\ h$ if $n<0 \end{cases}$.
\end{center}

\noindent When $\sigma$ appears in the definition of $J_{\alpha +1} = E(J_{\alpha})$ we produce an element $t_{-1}$ for which $t_{-1} h t_{-1}^{-1} = g$.

Finally we show that $G$ is strongly bounded.  The check will follow somewhat along the lines of the proof of \cite[Theorem 3.1]{dC}.  Suppose to the contrary, so that there exists an unbounded function $\Lambda: G \rightarrow \omega$ as in Proposition \ref{criterionforstrongboundedness}.  Select a sequence $\{g_n\}_{n\in \omega}$ of nontrivial elements in $G$ for which $\Lambda(g_n) \geq n^2$.  By how $G$ is constructed we may select $\alpha < \aleph_1$ for which $\{g_n\}_{n\in \omega} \subseteq K_{\alpha}$.  Let $\sigma: \mathbb{Z} \rightarrow K_{\alpha}\setminus \{1\}$ be given by

\begin{center}  $\sigma(z) = \begin{cases}g_n$ if $z = n\geq 0\\g_0 $ if $z<0\end{cases}$.
\end{center}

 In $K_{\alpha + 1} = E(K_{\alpha})$ there exists a collection of elements $\{t_z\}_{z\in \mathbb{Z}}$ and $t$ for which for all $z\in \mathbb{Z}$ we have $t_z\sigma(z)t_z^{-1} = \sigma(z + 1)$ and $tt_zt^{-1} = t_{z+1}$.  Select $M\in \omega$ large enough that $\Lambda(1), \Lambda(g_0), \Lambda(t_0), \Lambda(t_0^{-1}), \Lambda(t), \Lambda(t^{-1}) \leq M$.  We clearly have $\Lambda(t^{n}), \Lambda(t^{-n})\leq M + n$ for all $n\in \omega$.  Thus

\begin{center}  $\Lambda(t_n^{\pm 1}) = \Lambda(t^nt_0^{\pm 1}t^{-n}) \leq M + n + 2$
\end{center}

\noindent for all $n\in \omega$, from which we have by induction for $n \geq 1$ that

\begin{center}  $\Lambda(g_n) = \Lambda(t_{n-1}g_{n-1}t_{n-1}^{-1}) \leq M + 2n +1$

\end{center} 

\noindent and thus $n^2 \leq \Lambda(g_n)\leq M + 2n + 1$ for all $n\in\omega$, a contradiction.
\end{proof}

\end{section}

\begin{section}{Concluding Remarks}\label{concludingremarks}

Some questions remain regarding the set theoretic strength of the local-to-global bi-orderability theorem, which we'll denote $\LG$, mentioned in the introduction.  Since $\LG$ follows from the ultrafilter lemma, one naturally asks the following.

\begin{question}
Is $\LG$ strictly weaker than the ultrafilter lemma?  In other words, is there a model of $\ZF$ in which $\LG$ holds and the ultrafilter lemma is false?
\end{question}

\noindent Notice that $\LG$ implies the \emph{ordering principle} (every set can be given a total order).  To see, this one takes $X$ to be any set and considers the free group $\F(X)$.   Each finitely generated subgroup includes into some $\F(X')\leq \F(X)$ with $X' \subseteq X$ finite.  This $\F(X')$ is bi-orderable by applying the construction in \cite{M}, using the fact that $X'$ is finite and can therefore be given a total order.  Thus $\F(X)$ is locally bi-orderable, so $\F(X)$ is bi-orderable by $\LG$.  A bi-order restricts to a total order on the set $X$ of free generators.

\begin{question} Is $\LG$ strictly stronger than the ordering principle?
\end{question}

\noindent Since the ultrafilter lemma is strictly stronger than the ordering principle \cite{Mat}, the answer to at least one of the two above questions is ``yes''.

Finally, it would be interesting to produce an infinite bi-orderable strongly bounded group.  This task appears far more delicate than in the locally indicable case.  One must have unique root extraction in a bi-orderable group ($g^n = h^n$ implies $g = h$ whenever $n\geq 1$).  Our approach in the proof of Theorem \ref{locallyindicable} was quite heedless of such requirements on roots.

\end{section}

\section*{Acknowledgement}

The author is deeply grateful to Yago Antol\'in for introduction to, and many helpful conversations regarding, group orders.

\end{document}